\documentclass{llncs}
\usepackage{amsmath,amssymb,amsbsy,amsfonts,latexsym,
               amsopn,amstext,amsxtra,euscript,amscd}

\usepackage{todonotes}

\newcommand{\Z}{{\mathbb Z}}

\newcommand{\F}{{\mathbb F}}

\newcommand{\cB}{{\mathcal B}}

\newcommand{\beq}{\begin{equation}}

\let\workingver=n
\def\begeq#1{\begin{equation}\mylabel{#1}}
\def\endeq{\end{equation}}
\def\begalg{\begin{alg}}
\def\endalg{\end{alg}}

\def\refeq#1{\if\workingver y(\ref{#1})-[[#1]]\else(\ref{#1})\fi}
\def\refth#1{\if\workingver y\ref{#1}-[[#1]]\else\ref{#1}\fi}
\def\mylabel#1{\if\workingver y\label{#1}{\bf\ \ [[#1]]\ \ }
\else\label{#1}\fi}
\def\mybibitem#1{\if\workingver y\bibitem{#1}{\bf\ \ [[#1]]\ \ }
\else\bibitem{#1}\fi}

\def\uu{{\mathbf u}}
\def\ww{{\mathbf w}}
\def\vv{{\mathbf v}}
\def\xx{{\mathbf x}}

\def\00{{\mathbf 0}}

\newcommand{\V}{\mathbb{V}}
\newcommand{\BBF}{\mathbb{F}}

\def\cB{{\mathcal B}}

\def\uu{{\bf u}}
\def\vv{{\bf v}}
\def\ww{{\bf w}}
\def\xx{{\bf x}}

\def\00{{\bf 0}}
\def\11{{\bf 1}}
\def\+{\oplus}
\def\ss{{\bf s}}

\def \F {{\mathbb F}}

\def \Z {{\mathbb Z}}

\def \V {{\mathbb V}}

\def\W{\mathcal{W}}

\begin{document}
\newtheorem{thm}{Theorem}

\newtheorem{ex}[thm]{Example}

\title{\bf A complete characterization of plateaued Boolean  functions in terms of their Cayley graphs}
\titlerunning{A complete characterization of plateaued Boolean  functions in terms of their Cayley graphs}
\author{Constanza Riera\inst{1}, Patrick Sol\'e\inst{2}, Pantelimon St\u anic\u a\inst{3}}
\institute{Department of Computing, Mathematics, and Physics,\\
 Western Norway University of Applied Sciences \\
5020 Bergen, Norway; 
\email{\tt csr@hvl.no}
\and
 CNRS/LAGA, University of Paris 8, 2 rue de la Libert\'e,\\
 93 526 Saint-Denis, France; 
 \email{\tt patrick.sole@telecom-paristech.fr}, 
\and
Department of Applied Mathematics, Naval Postgraduate School,\\
Monterey, CA 93943, USA; 
\email{\tt pstanica@nps.edu}}
\authorrunning{C. Riera, P. Sol\'e, P. St\u anic\u a}
\toctitle{A complete characterization of plateaued Boolean  functions in terms of their Cayley graphs }
\tocauthor{C. Riera, P. Sol\'e, P. St\u anic\u a}

\maketitle

\baselineskip=1.1\baselineskip

\begin{abstract}
In this paper we find a complete characterization of plateaued Boolean  functions in terms of  the associated Cayley graphs. Precisely, we show that a Boolean function $f$ is $s$-plateaued (of weight $=2^{(n+s-2)/2}$) if and only if the associated Cayley graph is  a complete bipartite graph between the support of $f$ and its complement (hence the graph is strongly regular of parameters $e=0,d=2^{(n+s-2)/2}$). Moreover, a Boolean function $f$ is $s$-plateaued (of weight $\neq 2^{(n+s-2)/2}$) if and only if the associated Cayley graph is strongly $3$-walk-regular (and also strongly $\ell$-walk-regular, for all odd $\ell\geq 3$) with some explicitly given parameters.
\end{abstract}


\section{Introduction}

Boolean functions are very important objects in cryptography, coding theory, and communications, and have connections with many areas of discrete mathematics~\cite{CH1,CS17}.
In particular bent functions, which offer optimal resistance to linear cryptanalysis, when used in symmetric cryptosystems, have been extensively studied~\cite{MesnagerBook,Tok15}.
They were shown in~\cite{BC1,BCV} to be connected to strongly regular graphs. This connection occurs through the Cayley graph with generator set the support of the 
Boolean function (denoted by $\Omega_f$ below). Namely, having two nonzero components in the Walsh-Hadamard spectrum translates at the Cayley graph level as having three eigenvalues.
This link is often referred to as the {\em Bernasconi-Codenotti correspondence}. 

In this paper, we extend this connection by relating semibent and, in general, plateaued functions with a special class of walk-regular graphs. Plateaued Boolean functions are characterized as having
three values in their Walsh-Hadamard spectrum~\cite{Mes14}. Their corresponding Cayley graphs belong to a special class of regular graphs with either three or four eigenvalues in their spectrum.
The three eigenvalue case is dealt with by the strong regularity and the four eigenvalues case corresponds to the strongly $t$-walk-regular graphs introduced
by Fiol and Garriga~\cite{FG07}. The special case of  four  eigenvalues of these graphs was studied in particular in~\cite{DO13}.

The material is organized as follows. The next section compiles the necessary notions and definitions on Boolean functions and graph spectra.
Section 3 derives the main characterization result of the paper.

\section{Preliminaries}
\subsection{Boolean functions}
Let $\F_2$ be the finite field with two elements and $\Z$ be the ring of integers.
For any $n \in \Z^+$, the set of positive integers, let
$[n] = \{1, \ldots, n\}$. The Cartesian product of $n$ copies of $\F_2$ is
$\F_2^n =\{\xx = (x_1, \ldots, x_n) : {x_i \in \F_2, i \in [n]}\}$ which is
an $n$-dimensional vector space over $\F_2$, which we will denote by $\V_n$.
We will denote by $\oplus$, respectively, $+$, the  operations on $\F_2^n$, respectively, $\Z$.
For any $n \in \Z^+$, a function $F : \V_n \rightarrow \F_2$
is said to be a {\em  Boolean function} in $n$ variables. The set of all Boolean functions will be denoted by $\mathcal{B}_n$.
A Boolean function can be regarded as a multivariate polynomial over $\BBF_{2}$, called the {\em algebraic normal form}~(ANF)
\[
f(x_{1},\ldots,x_{n})=a_{0}\+ \sum_{1\leq i\leq n}a_{i}x_{i}\+ \sum_{1\leq i<j\leq n}a_{ij}x_{i}x_{j}\+ \cdots\+  a_{12\ldots n}x_{1}x_{2}\ldots x_{n},
\]
where the coefficients $a_{0},\,a_i,\,a_{ij},\,\ldots,\,a_{12\ldots n}\in\BBF_{2}$.
 The maximum number of variables in a monomial is called the ({\em algebraic}) {\em degree}.

For a Boolean function $f\in\cB_n$, we define its sign function $\hat f$ by $\hat f(\xx)=(-1)^{f(\xx)}$.
For $\uu=(u_1,\ldots,u_n)$, $\xx=(x_1,\ldots,x_n)$, we let $\uu\cdot\xx=\sum_{i=1}^nu_ix_i$ be the regular scalar (inner) product on $\V_n$. For a binary string $\ss$, we let $\bar \ss$ denote the binary complement of $\ss$. The (Hamming) {\em weight} of a binary string $\ss$, denoted by $wt(\ss)$, is the number of nonzero bits in $\ss$.

We order $\BBF_2^n$ lexicographically, and denote $\vv_0=(0,\ldots,0,0)$,
$\vv_1=(0,\ldots,0,1)$, $\vv_{2^n-1}=(1,\ldots,1,1)$.
The {\em  truth table} of a Boolean function $f\in\cB_n$ is the binary string of length $2^n$,
$[f(\vv_0),$ $f(\vv_1),   \ldots, f(\vv_{2^n-1})]$ (we will often omit the commas). The (Hamming) {\em weight} of a function $f$ is the cardinality of the support $\Omega_f=\{\xx\,:\, f(\xx)=1\}$, that is, is the weight of its truth table.
We define the {\em Fourier transform} of $f$  by
\[
\mathcal{W}_f(\uu) = \sum_{\xx\in \V_n}{f(\xx)}(-1)^{\uu \cdot \xx },
\]
and the {\em Walsh-Hadamard transform} of $f$ by
\[
\mathcal{W}_{\hat{f}}(\uu) = \sum_{\xx\in \V_n}(-1)^{f(\xx)}(-1)^{\uu \cdot \xx }.
\]
A function $f$ for which $|\mathcal{W}_{\hat{f}}(\uu)| = 2^{n/2}$ for all $\uu\in \V_n$ is called a {\em bent} function~\cite{RO76}. Further recall that $f\in\mathcal{B}_n$ is called {\em plateaued} if
$|\mathcal{W}_{\hat{f}}(\uu)| \in \{0,2^{(n+s)/2}\}$ for all $\uu\in \V_n$ for a fixed integer $s$ depending on $f$ (we also call $f$ then $s$-{\em plateaued}). If $s=1$ ($n$ must then be odd), or $s=2$ ($n$ must then be even), we call $f$ {\em semibent}.
For more on Boolean functions (bent, semibent, plateaued, etc.), the reader can consult~\cite{Bud14,CH1,CS17,MesnagerBook} and the references therein.

\subsection{A short primer on  strong regularity and walk regularity}

A graph is {\em regular of degree $r$}  (or $r$-regular) if every vertex has degree $r$, where the degree of a vertex is defined as the number of edges incident to it.  We say that an $r$-regular graph $G$ with $v$ vertices is a {\em strongly regular graph} (srg) with parameters $(v,r,e,d)$ if
 there exist nonnegative integers $e,d$ such that for all vertices $\uu,\vv$
the number of vertices adjacent to both $\uu,\vv$ is $e$, (resp. $d$),  if $\uu,\vv$ are
adjacent, (resp. nonadjacent). See~\cite{CDS} for further properties of these graphs.

For a Boolean function $f$ on $\V_n$, we define the {\em Cayley graph} of $f$ to be the graph $G_f=(\V_n,E_f)$
 whose vertex set is $\V_n$, and whose set of edges is defined by
\[
E_f=\{(\ww,\uu)\in \V_n\times \V_n\,:\, f(\ww\oplus \uu)=1\}.
\]

The adjacency matrix $A_f$ is the matrix whose entries are $A_{i,j}=f({\bf i}\oplus {\bf j})$ (where ${\bf i}$ is the binary representation as an $n$-bit vector of the index $i$).
It is simple to prove that $A_f$ has the dyadic property: $A_{i,j}=A_{i+2^{k-1},j+2^{k-1}}$.
One can derive from its definition that $G_f$ is a {\em regular graph of degree
$wt(f)=|\Omega_f|$} (see \cite[Chapter 3]{CDS} for further definitions and properties of these graphs).

Given a graph, $G_f$, and its adjacency matrix, $A$, the {\em spectrum} $Spec(G_f)$ is the
 set of eigenvalues of $A$ (called also the eigenvalues of $G_f$).
We assume throughout that $G_f$ is connected (in fact, one can show that all connected components of $G_f$ are isomorphic)~\cite{BC1,CDS}.

  It is known (see \cite[pp. 194--195]{CDS})
that a connected $r$-regular graph is strongly regular if and only if it has exactly
three distinct eigenvalues
$\lambda_0=r,\lambda_1,\lambda_2$ (so $e=r+\lambda_1\lambda_2+\lambda_1+\lambda_2$,
$d=r+\lambda_1\lambda_2$).
Bent functions exactly correspond to those strongly regular graphs with $e=d$ (Bernasconi-Codenotti correspondence).

The following result is known \cite[Th. 3.32, p. 103]{CDS} (the second part follows from a counting argument and is also well known).
\begin{proposition}
\label{prop1}
If $A$ is the adjacency matrix of a strongly $r$-regular graph of parameters $e,d$ and $|V|=v$, then
\[
A^2=(e-d)A+(r-d)I+dJ,
\]
where $J$ is the all $1$ matrix. Further, $r(r-e-1)=d(v-r-1).$
\end{proposition}

The distance in the graph $\Gamma=(V,E)$ between two vertices $x, y\in V$, denoted by $d(x, y)$,
is given by the length of the shortest path between $x$ and $y$. The diameter of a graph is $D=\max_{x,y\in V} d(x,y)$.
A connected graph is called  {\em distance-regular} of parameters $(c_i,a_i,b_i)$ (called intersection numbers), if, for all $0\leq i\leq D$, and for all vertices $x,y$ with $d(x,y)=i$, among the neighbors of $y$, there are $c_i$  that are at distance $i-1$ from $x$, $a_i$ at distance $i$, and $b_i$ at distance $i+1$ (thus $\Gamma$ is regular of degree $r=b_0$).

 Fiol and Garriga~\cite{FG07}  introduced $t$-walk-regular graphs as a generalization of both
distance-regular and walk-regular graphs. We call a graph $\Gamma=(V,E)$  a {\em $t$-walk-regular} (assuming $\Gamma$ has its diameter   at least $t$) if the number
of walks of every given length $\ell$ between two vertices $x, y\in V$ depends only on the
distance between $x,y$, provided it is $\leq t$. In \cite{DO13}, van Dam and Omidi generalized this concept and called $\Gamma$ a 
{\em strongly $\ell$-walk-regular} with parameters $(\sigma_\ell,\mu_\ell,\nu_\ell)$ if there are $\sigma_\ell,\mu_\ell,\nu_\ell$ walks of length $\ell$ between every two adjacent, every two non-adjacent, and every two identical vertices, respectively. Certainly, every strongly regular graph of parameters $(v,r,e,d)$ is a strongly $2$-walk-regular graph with parameters $(e,d,r)$.

Similarly to Proposition~\ref{prop1}, the adjacency matrix $A$ of a strongly $\ell$-walk-regular graph will satisfy the following property.
\begin{proposition}[\cite{DO13}]
Let $\ell>1$, and $A$ be the adjacency matrix of a graph $\Gamma$. Then $\Gamma$ is a strongly $\ell$-walk-regular with parameters $(\sigma_\ell,\mu_\ell,\nu_\ell)$ if and only if
\[
A^\ell+(\mu_\ell-\sigma_\ell) A+(\mu_\ell-\nu_\ell)I=\mu_\ell J.
\]
\end{proposition}

\section{Plateaued Boolean functions}

In general, the spectrum of the Cayley graph of an  $s$-plateaued Boolean function $f:\F_2^n\to\F_2$ will be 4-valued, and therefore the graph will not be  strongly regular (see  \cite[Theorem 9.7]{CS17}). This can be easily deduced from the fact that, if the Walsh-Hadamard transform of a Boolean function takes values in $\{0,\pm k\}$ (for $s$-plateaued functions, $k=2^{(n+s)/2}$), then the Fourier transform of $f$ takes values in $\{wt(f),0,\pm \frac{k}{2}\}$ (recall that the Fourier transform of $f$ gives the graph spectrum of the corresponding Cayley graph), as the following argument shows.

By  \cite[Eq. (2.15)]{CS17}, $$\W_f(\ww)=2^{n-1}\delta(\ww)-\frac{1}{2}\W_{\hat{f}}(\ww),$$
where $\delta$ is the Kronecker delta.
Note that, for $\ww=\00, \W_f(\00)=wt(f)$.
By Parseval's identity (see~\cite{CS17}), $\displaystyle 2^{2n}=\sum_{\ww\in\F_2^n} |\W_{\hat{f}}(\ww)|^2$, the multiplicity of $\pm k$ is $\frac{2^{2n}}{k^2}$. Hence, the multiplicity of these eigenvalues will be (assuming $wt(f)\neq \frac{k}{2}$; the other case follows easily):
\begin{enumerate}
\item[$(i)$] If $f$ is balanced, then $\W_{\hat{f}}(\00)=0$, while $\W_f(\00)=wt(f)$.  Then, the multiplicity of $\lambda_1=wt(f)$ is 1, the multiplicity of $\lambda_3=0$ is $2^n-\frac{2^{2n}}{k^2}-1$, while the multiplicities of $\lambda_{2},\lambda_4=\pm \frac{k}{2}$ will sum to $\frac{2^{2n}}{k^2}$.
\item[$(ii)$]  If $f$ is not balanced, then $\W_{\hat{f}}(\00)=\pm k$, while $\W_f(\00)=wt(f)$. Then, the multiplicity of $\lambda_1=wt(f)$ is 1, the multiplicity of 0 is $2^n-\frac{2^{2n}}{k^2}$, while the multiplicities of $\pm \frac{k}{2}$ will sum to $\frac{2^{2n}}{k^2}-1$.
\end{enumerate}

{\bf Example}: $n=3$, $f=x_1x_2\oplus x_1x_3\oplus x_2x_3$, which is semibent, since $\W_{\hat{f}}(\ww)=(0\ 4 \ 4 \ 0 \ 4\ 0\ 0\ -4))^T$. We compute that  $\W_f(\ww)=(4\ -2\ -2\ 0 \ -2\ 0\ 0\ 2)^T$, which is 4-valued.

Certainly, if $f$ is semibent, the multiplicities are more precisely known (see~\cite{Mes14}, for example).
For instance, if $n$ is odd (without loss of generality, we assume that $f(\00)=0$),  the multiplicities of the spectra coefficients of $\hat f$ are
\begin{align*}
{\text{value}} \qquad & \qquad {\text {multiplicity}}\\
0  \qquad & \qquad  2^{n-1}\\
2^{(n+1)/2}  \qquad & \qquad  2^{n-2}+2^{(n-3)/2}\\
-2^{(n+1)/2}  \qquad & \qquad  2^{n-2}-2^{(n-3)/2}.
\end{align*}

We show in Figure~\ref{semibent1} the Cayley graph of a semibent function.
\begin{figure}[h]
\begin{center}
\includegraphics[width=.6\textwidth]{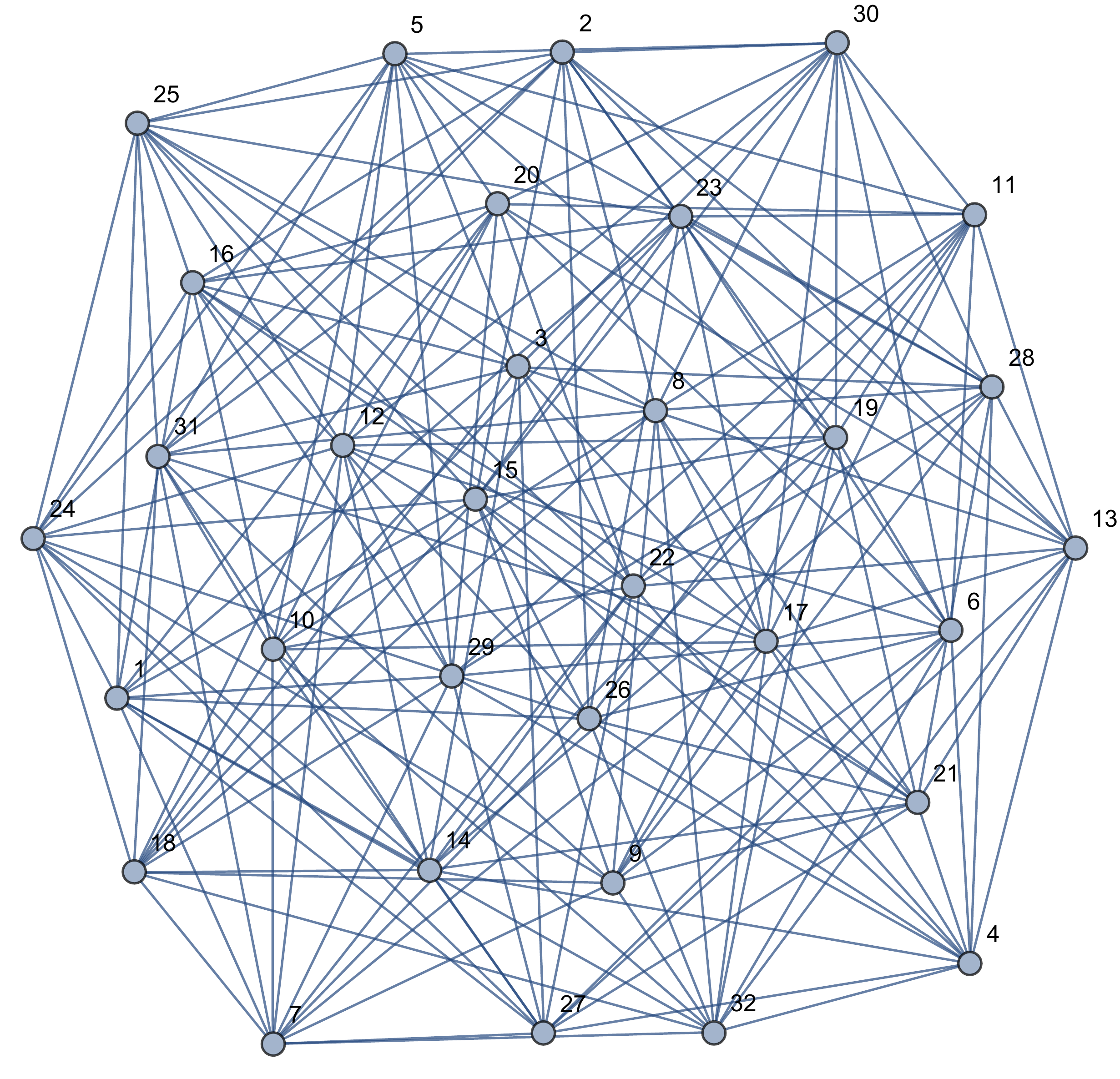}
\end{center}
\caption{Cayley graph associated to the semibent $f(\xx)=x_1 x_2 \+ x_3 x_4 \+ x_1 x_4 x_5  \+ x_2 x_3 x_5 \+ x_3 x_4 x_5$}
\label{semibent1}
\end{figure}

\subsection{$s$-Plateaued Boolean functions $f$ with $wt(f)=2^{(n+s-2)/2}$}

\begin{thm}
If $f:\F_2^n\to \F_2$ is $s$-plateaued and $wt(f)=2^{(n+s-2)/2}$, then  $G_f$ (if connected) is the complete bipartite graph between the vectors in $\Omega_f$ and vectors in $\F_2^n\setminus \Omega_f$ (if disconnected, it is a union of complete bipartite graphs). Moreover, $G_f$ is a strongly regular graph with $(e,d)=\left(0,2^{(n+s-2)/2}\right)$.
\end{thm}
\begin{proof}
We know that the Walsh-Hadamard spectra of $\hat f$ in this case is $\{0,\pm 2^{(n+s)/2}\}$ and therefore, the spectra of $f$ is
also 3-valued, that is, $\{ wt(f),0,\pm 2^{(n+s-2)/2}\}=\{0,\pm 2^{(n+s-2)/2}\}$, and thus, the Cayley graph of $f$ in this case is strongly regular.  Now, from~\cite{CDS}, we know that if $G_f$ has three distinct eigenvalues $\lambda_0=wt(f)>\lambda_1=0>\lambda_2=-\lambda_0$
, then $G_f$ is the complete bipartite graph between the nodes in $\Omega_f$ and nodes in $\F_2^n\setminus \Omega_f$.

Since the eigenvalues of the strongly regular graph $G_f$ of $f$ can be expressed in terms of the parameters $e,d$, namely
\begin{align*}
\lambda_0=wt(f),\
\lambda_{1,2}=\frac{1}{2}\left(e-d\pm \sqrt{(e-d)^2-4(d-wt(f)}) \right),
\end{align*}
or equivalently, $e=r+\lambda_1\lambda_2+\lambda_1 +\lambda_2,d=r+\lambda_1\lambda_2$, and given the Walsh-Hadamard spectra of $f$,  the last claim follows.
\qed
\end{proof}

\subsection{General $s$-plateaued Boolean functions}

We now assume that $f$ is $s$-plateaued and $wt(f)\neq 2^{(n+s-2)/2}$,  and, therefore, the spectrum of $G_f$ is 4-valued.
It is known (see~\cite{Huang}) that if $G$ is connected and regular with four distinct eigenvalues, then $G$ is walk-regular. In fact, in our case a result much stronger is true (see our theorem below).
We will need the following two propositions (we slightly change notations, to be consistent).
\begin{proposition}[van Dam and  Omidi \textup{\cite[Proposition 4.1]{DO13}}]
\label{propDO1}
 Let $\Gamma$ be a connected regular graph with four distinct eigenvalues $r > \lambda_2 > \lambda_3 > \lambda_4$. Then $\Gamma$
is strongly $3$-walk-regular if and only if $\lambda_2+\lambda_3+\lambda_4 = 0$.
\end{proposition}
\begin{proposition}[van Dam and  Omidi \textup{\cite[Proposition 3.1]{DO13}}]
\label{propDO2}
 A connected $r$-regular graph $\Gamma$ on $v$ vertices is strongly $\ell$-walk-regular with parameters
$(\sigma_\ell,\mu_\ell,\nu_\ell)$ if and only if all eigenvalues except $r$ are roots of the equation
\[
x^\ell+(\mu_\ell-\sigma_\ell)x+\mu_\ell-\nu_\ell=0,
\]
and $r$ satisfies
\[
r^\ell+(\mu_\ell-\sigma_\ell)r+\mu_\ell-\nu_\ell=\mu_\ell v.
\]
\end{proposition}
In our main theorem of this section we show the counterpart for the Bernasconi-Codenotti equivalence in the case of plateaued functions.
\begin{thm}
\label{thm1}
Let $f:\F_2^n\to\F_2$ be a Boolean function, and assume that $G_f$ is connected, and that $r:=wt(f)\neq 2^{(n+s-2)/2}$.
Then, $f$ is $s$-plateaued (with $4$-valued spectra for $f$) if and only if $G_f$ is strongly $3$-walk-regular of parameters $(\sigma,\mu,\nu)=(2^{-n}r^3+2^{n+s-2} -2^{s-2}r,2^{-n}r^3-2^{s-2}r,2^{-n}r^3-2^{s-2}r)$ (hence $\mu=\nu$).
\end{thm}
\begin{proof}
We first assume that $f$ is $s$-plateaued and so, its spectra is $\{0,\pm 2^{(n+s)/2}\}$. Consequently, the spectra of $G_f$ is $4$-valued (since $r:=wt(f)\neq 2^{(n+s-2)/2}$), namely $\{r=wt(f),\lambda_2:=2^{(n+s-2)/2},\lambda_3:=0,\lambda_4:=-2^{(n+s-2)/2}\}$. The fact that $G_f$ is strongly 3-walk-regular follows from  Proposition~\ref{propDO1}, since $\lambda_2+\lambda_3+\lambda_4=0$, which certainly happens for our graphs. Moreover, the parameters $(\sigma,\mu,\nu)$ (we removed, for convenience, the subscripts $\ell=3$) can be found using Proposition~\ref{propDO2} as solutions to the diophantine system (recall that in our case $v=2^n$ and $r=wt(f)$)
\begin{align*}
0&=2^{3(n+s-2)/2}+(\mu-\sigma)2^{(n+s-2)/2}+\mu-\nu,\\
0&=-2^{3(n+s-2)/2}-(\mu-\sigma)2^{(n+s-2)/2}+\mu-\nu,\\
\mu\, 2^n&=r^3+(\mu-\sigma)r+\mu-\nu,
\end{align*}
namely, $(\sigma,\mu,\nu)=(2^{-n}r^3+2^{n+s-2} -2^{s-2}r,2^{-n}r^3-2^{s-2}r,2^{-n}r^3-2^{s-2}r)$.

Conversely, assuming $G_f$ is a $3$-walk-regular graph with the above parameters, then the eigenvalues $\lambda_2>\lambda_3>\lambda_4$ will satisfy the equation
\begin{align*}
x^3+(\mu-\sigma)x+\mu-\nu=0,
\end{align*}
which will render the roots, $\lambda_2=2^{(n+s-2)/2},\lambda_3=0,\lambda_4=-2^{(n+s-2)/2}$. The claim is shown.
\qed
\end{proof}

\begin{remark}
Using a result of Godsil~\textup{\cite{Go88}} one can easily show (under mild conditions -- thus removing strongly regular ones, for example) that the graphs corresponding to plateaued functions are not distance-regular.
\end{remark}

In fact, from~\cite{DO13} we know that the graph with four distinct eigenvalues is $\ell$-walk-regular for any odd $\ell\geq 3$, but in our case we can show a lot more, by finding the involved parameters precisely.

\begin{thm}
If $A$ is the adjacency matrix of the Cayley graph corresponding to an $s$-plateaued with $4$-valued spectra (of $f$), then $G_f$ is strongly $\ell$-walk-regular for any odd $\ell$ of parameters $(\sigma_\ell,\mu_\ell,\nu_\ell)$, where  $\ell=2t+1$, $\sigma_\ell=\mu \frac{2^{(n+s-2)t}-r^{2t}}{2^{n+s-2}-r^2}+2^{(n+s-2)t}$,  $\mu_\ell=\nu_\ell=\mu \frac{2^{(n+s-2)t}-r^{2t}}{2^{n+s-2}-r^2}$.  Further,  the following identity holds, for all $t\geq 1$,
\[
A^{2t+1}=2^{(n+s-2)t}A+ \mu \frac{2^{(n+s-2)t}-r^{2t}}{2^{n+s-2}-r^2}\,J\enspace,
\]
where $(\sigma,\mu,\nu)=(2^{-n}r^3+2^{n+s-2} -2^{s-2}r,2^{-n}r^3-2^{s-2}r,2^{-n}r^3-2^{s-2}r)$.
\end{thm}
\begin{proof}
From our Theorem~\ref{thm1}, we know that
\[
A^3=(\sigma-\mu)A+\mu J,
\]
since we know that $\mu=\nu$. We will show our result by induction, and so, for simplicity we label $x_1:=\sigma-\mu=2^{n+s-2},y_1:=\mu=2^{-n}r^3-2^{s-2}r$.
Assume now that
\begin{equation}
\label{eq1}
A^{2t+1}=x_t A+y_t J.
\end{equation}
First, observe that, since our graph is regular of degree $r$, then $AJ=rJ$, and more general, $A^k J=r^k J$.
Multiplying~\eqref{eq1} by $A^2$, we get
\begin{align*}
A^{2t+3}
&=x_t A^3+y_t A^2 J\\
&=x_t (x_1 A+y_1 J)+y_t r^2 J\\
&=x_tx_1 A+(x_ty_1+y_tr^2)J,
\end{align*}
and consequently, we get the recurrences
\begin{align*}
x_{t+1}&=x_tx_1\\
y_{t+1}&=x_ty_1+y_tr^2.
\end{align*}
\begin{sloppypar}
Solving the system, we get $x_{t+1}=x_1^{t+1}=(\sigma-\mu)^{t+1}=2^{(n+s-2)(t+1)}$ and $\displaystyle y_{t+1}=y_1\frac{x_1^{t+1}-r^{2(t+1)}}{x_1-r^2}=\mu \frac{2^{(n+s-2)(t+1)}-r^{2(t+1)}}{2^{n+s-2}-r^2}$, and our claim is shown.
\end{sloppypar}
\qed
\end{proof}

\end{document}